\documentclass[11pt]{article}

\usepackage[left=32mm,right=32mm,top=30mm,bottom=32mm]{geometry}

\usepackage{caption}

\usepackage{xfrac}

\usepackage[percent]{overpic}
\usepackage{ulem}
\usepackage{color}
\usepackage{amsthm,amsmath,amssymb}
\usepackage{booktabs}
\usepackage{mathpazo}
\usepackage{microtype}
\usepackage{overpic}
\usepackage{bm}
\usepackage{sectsty}
\usepackage[
	pdftitle={PDFTitle},
	pdfauthor={Hugo Parlier},
	ocgcolorlinks,
	linkcolor=linkred,
	citecolor=linkred,
	urlcolor=linkblue]
{hyperref}

\usepackage{mathtools}

\definecolor{linkred}{RGB}{199,21,133}
\definecolor{linkblue}{RGB}{16, 78, 139}

\usepackage[hang,flushmargin]{footmisc}
\usepackage{enumitem}
\usepackage{titlesec}
	\titlespacing{\section}{0pt}{12pt}{0pt}
	\titlespacing{\subsection}{0pt}{6pt}{0pt}
	
\titlelabel{\thetitle.\quad}

\makeatletter 

\long\def\@footnotetext#1{%
\H@@footnotetext{%
\ifHy@nesting 
\hyper@@anchor{\@currentHref}{#1}%
\else 
\Hy@raisedlink{\hyper@@anchor{\@currentHref}{\relax}}#1%
\fi 
}}

\def\@footnotemark{%
\leavevmode 
\ifhmode\edef\@x@sf{\the\spacefactor}\nobreak\fi 
\H@refstepcounter{Hfootnote}%
\hyper@makecurrent{Hfootnote}%
\hyper@linkstart{link}{\@currentHref}%
\@makefnmark 
\hyper@linkend 
\ifhmode\spacefactor\@x@sf\fi 
\relax 
}%

\ifFN@multiplefootnote%
\renewcommand*\@footnotemark{%
\leavevmode 
\ifhmode 
\edef\@x@sf{\the\spacefactor}%
\FN@mf@check 
\nobreak 
\fi 
\H@refstepcounter{Hfootnote}%
\hyper@makecurrent{Hfootnote}%
\hyper@linkstart{link}{\@currentHref}%
\@makefnmark 
\hyper@linkend 
\ifFN@pp@towrite 
\FN@pp@writetemp 
\FN@pp@towritefalse 
\fi 
\FN@mf@prepare 
\ifhmode\spacefactor\@x@sf\fi 
\relax%
}%
\fi 

\makeatother

\newtheorem{thm}{Theorem}

\newtheorem{coro}[thm]{Corollary}
\newtheorem{lem}[thm]{Lemma}
\newtheorem{prop}[thm]{Proposition}

\theoremstyle{definition}

\newtheorem{eg}[thm]{Example}

\theoremstyle{remark}

\newcommand{\Hess}{{\rm Hess}}

\newcommand{\oP}{\overline{P}}

\newcommand{\oU}{\overline{U}}

\newcommand{\ocrit}{\overline{{\rm crit}}}

\newcommand{\crit}{{\rm crit}}
\newcommand{\sing}{{\rm sing}}
\newcommand{\dist}{{\rm dist}}

\newcommand{\tu}{\widetilde{u}}

\newcommand{\Rbb}{ {\mathbb R}}
\newcommand{\Zbb}{ {\mathbb Z}}
\newcommand{\Nbb}{ {\mathbb N}}
\newcommand{\Cbb}{ {\mathbb C}}

\renewcommand{\phi}{\varphi}

\definecolor{blue(ncs)}{rgb}{0.0, 0.53, 0.74}

\definecolor{violet}{rgb}{0.56, 0.0, 1.0}
\definecolor{olive}{rgb}{0.5, 0.5, 0.0}

\sectionfont{\large \bfseries}
\subsectionfont{\normalsize}

\long\def\symbolfootnote[#1]#2{\begingroup%
\def\thefootnote{\fnsymbol{footnote}}\footnote[#1]{#2}\endgroup}

\def\blfootnote{\xdef\@thefnmark{}\@footnotetext}

\date{\today}

\begin{document}

{\Large \bfseries 
Some remarks on critical sets of Laplace eigenfunctions}

{\large 
Chris Judge\symbolfootnote[2]{
Research partially supported by a Simons collaboration grant.} 
and Sugata Mondal\symbolfootnote[1]{
partially supported by Ramanujan Fellowship of SERB, Govt. of India. \vspace{.1cm} 
}
}

\vspace{0.5cm}

{\bf Abstract.} We study the set of critical points of a solution to 
$\Delta u = \lambda \cdot u$ and in particular components of the critical
set that have codimension 1. We show, for example, 
that if a second Neumann eigenfunction of a simply
connected polygon $P$ has infinitely many critical points, 
then $P$ is a rectangle. 

\vspace{1cm}

\section{Introduction}

Let $u: \Omega \to \Rbb$ satisfy $\Delta u = \lambda \cdot u$ where 
$\Omega \subset \Rbb^d$ is an open set, $\Delta$ is the Laplacian, 
and $\lambda \in \Rbb$. In this note, we study the critical set 
$\crit(u):=|\nabla u|^{-1}(0)$ of $u$ and in particular the consequences of 
having a hypersurface contained in $\crit(u)$. 
For example, we show that if a connected component $A$ of $\crit(u)$
contains a hypersurface, then $A$ is a proper smooth hypersurface
(Theorem \ref{thm:crit-structure-higher}), and we show that 
if $A$ is contained in a hyperplane or sphere, then $u$ depends only on the 
distance to the hyperplane or sphere (Theorem \ref{thm:invariant}).
We then derive consequences for second Neumann eigenfunctions 
on simply connected domains. For example, in \S \ref{sec:2nd-Neumann}
we prove the following:

\begin{thm}[Compare Theorem \ref{thm:rectangle-interior} below]
\label{thm:rectangle}
Let $P \subset \Rbb^2$ be a bounded, simply connected, polygonal domain, 
and let $u: P \to \Rbb$ be a second Neumann eigenfunction of the Laplacian. 
If the set of critical points lying in $\oP$ is not finite, 
then $P$ is a rectangle and $u$ is a multiple of 
$x \mapsto \cos(\sqrt{\lambda} \cdot \dist(x,e))$ where $e$
is a side of $P$.
\end{thm}

More generally, we show that if $u: \Omega \to \Rbb$ is a
second Neumann eigenfunction on a bounded, simply connected, 
open     set $\Omega \subset \Rbb^d$ with Lipschitz boundary, 
then $\crit(u) \cap \Omega$ 
does not contain a hypersurface (Proposition \ref{coro:lem-2nd-no-hypersurface}). 
One wonders whether there exists a non-simply connected domain $\Omega \subset \Rbb^d$ 
whose second Neumann eigenfunction contains a critical hypersurface.

Many of the results described here may be formulated so as to
apply to Laplacians associated to a real-analytic metric on
a real-analytic manifold. Moreoever,  proofs of some 
results---for example 
Theorem \ref{thm:invariant} and Proposition \ref{coro:lem-2nd-no-hypersurface} ---extend to the case of smooth metrics.
It would be interesting to know exactly which results extend to 
smooth metrics. For example, Theorem \ref{thm:rectangle}
implies that a sequence of isolated critical points of $u$ does
limit to the boundary of the polygon $P$. Does there exist
a Lipschitz domain $\Omega \subset \Rbb^2$ and a Neumann eigenfunction
$u:\Omega \to \Rbb$ so that the $\crit(u)$ consists of infinitely 
many isolated points? Note that it was recently proved in 
\cite{BLS} that there exists a smooth metric on the 2-torus
so that the associated Laplacian has a sequence $u_n$ 
of eigenfunctions with eigenvalues tending to infinity 
so that $\crit(u_n)$ consists of infinitely many 
isolated points.

In connection to Proposition \ref{coro:lem-2nd-no-hypersurface} we would 
like to mention the {\it Schiffer's conjecture}. 
A variant of this conjecture says that a Neumann eigenfunction 
on a simply connected domain can have a loop in its critical set if and only 
if the domain is a disc and the loop is a distance circle \cite{W76}.


\section{Critical hypersurfaces}

The {\em singular set}, $\sing(u)$, consists of 
critical points $x$ of $u$ such that $u(x)=0$. 
The following is Lemma (1.9) in \cite{Hardt-Simon}
specialized to the case of the Euclidean Laplacian.
See also \cite{Caffarelli-Friedman}.

\begin{lem}
Let $\Omega \subset \Rbb^d$ be an open set and 
let $u: \Omega \to \Rbb$ be a nonconstant Laplace eigenfunction.
If $U$ is open and bounded with $\oU \subset \Omega$, then
$\sing(u) \cap U$ is contained
in the union of finitely many closed analytic submanifolds 
of codimension two.
\end{lem}

For the convenience of the reader 
we provide the proof of \cite{Hardt-Simon}.
Let $\partial_j$ denote the partial derivative $\partial/\partial x_j$,
and for each multi-index $\alpha \in \Nbb^d$, set
$\partial^{\alpha} := \partial_1^{\alpha_1} \cdots \partial_d^{\alpha_d}$.

\begin{proof}
Let $S_{d}$ be the set of points $x \in \Omega$
such that $\partial^{\alpha}u(x)= 0$ for each $|\alpha|< d$
and such that $\partial^{\alpha}u(x) \neq 0$ for some $\alpha$ 
with $|\alpha|=d$.
Since $u$ is real-analytic, $\Omega$ equals 
the disjoint union $\cup_{d=0}^{\infty} S_d$. 
Because $\partial^{\alpha}u(x)$ is continuous and 
$\oU \subset \Omega$ is compact, there exists
$d_U$ such that $S_d \cap U = \emptyset$ for each 
$d \geq d_U$. 

Each singular point lies in some $S_d$ with 
$2 \leq d \leq d_U$, and so to prove the claim, 
it suffices to show that for each $d \geq 2$,
the set $S_d$ is contained in finitely many 
codimension two submanifolds. 
If $x \in S_d$ with $d \geq 2$, then $u(x)=0$ 
and there exists $\alpha$ with $|\alpha|=d-2$
so that $\Hess(\partial^{\alpha}u)(x) \neq 0$.
On the other hand, 
$\Delta \partial^{\alpha} u(x) = \lambda \partial^{\alpha}u(x)=0$
and hence the trace of $\Hess(\partial^{\alpha}u)(x)=0$.
Therefore, the rank of $\Hess(\partial^{\alpha}u)(x)$
is at least two. In particular, there exist $j$ and $k$ 
so that the vectors $\nabla (\partial_j \partial^{\alpha} u)(x)$
and $\nabla (\partial_k \partial^{\alpha} u)(x)$ are nonzero
and linearly independent. The analytic implicit function theorem
then gives a neighborhood $V_x$ of $x$ such 
that $\partial_j \partial^{\alpha} u^{-1}(0) \cap
\partial_j \partial^{\alpha} u^{-1}(0) \cap V_x$ is a 
codimension two analytic submanifold of $U$. Since $\oU$
is compact, there exists a finite set $F \subset  U$ 
such that $U = \cup_{x \in F} V_x$, and hence finitely many
codimension two submanifolds cover $S_d$.
\end{proof}

The following is immediate.

\begin{coro}
\label{coro:no-intersect-nodal}
Let $\Omega \subset \Rbb^d$ be an open set and let 
$u: \Omega \to \Rbb$ be a nonconstant Laplace eigenfunction.
Suppose that $A$ is a connected component of $\crit(u)$.
If there exists an open set $U$ such that $U \cap A$
is a hypersurface, then $A$ does not intersect $u^{-1}(0)$.
\end{coro}

\vspace{.3cm}

The Cauchy-Kovalevskaya theorem implies that each solution
to $\Delta u = \lambda \cdot u$ is real-analytic.
Thus, each partial derivative $\partial_j u$ is real-analytic,
and therefore $\crit(u)$ is a real-analytic variety. 
In general, every real-analytic variety $S$ 
possesses a stratification. In particular, there 
exists a (possibly disconnected) $m$-dimensional real-analytic submanifold 
$R \subset S$ so that the Hausdorff dimension of $S \setminus R$
is at most $m-1$.\footnote{See, for example, paragraph 3.4.10 in \cite{Federer}.} 
In the case where $S$ is a connected component of  $\crit(u)$, 
the following implies that $S \setminus R$ is empty.

\begin{thm}
\label{thm:crit-structure-higher}
Let $\Omega \subset \Rbb^d$ be open, suppose that 
$u: \Omega \to \Rbb$ solves $\Delta u = \lambda \cdot u$ with $\lambda \neq 0$,
and let $A$ be a connected component of $\crit(u)$.
If there exists an open set $U \subset \Omega$ such that $U \cap A$
is a $d-1$ dimensional manifold, 
then $A$ is a proper real-analytic manifold of dimension $d-1$.
\end{thm}

\begin{proof}
Let $E$ denote the subset of $A$ consisting of $x$ that have a neighborhood 
$W$ such that $A \cap W$ is a smooth hypersurface. The set $E$
is open. Since $A$ is closed and connected, it suffices to 
prove that $E$ is also closed.

Let $y$ lie in the closure of $E$. Then $y \in A$,
and Corollary \ref{coro:no-intersect-nodal} implies that $u(y) \neq 0$.
Hence since $\lambda \neq 0$, we have $\Delta u(y) \neq 0$.
In particular, $\partial_j^2 u(y) \neq 0$ for some $j$.
Because $y \in A$, we have $\partial_j u(y)=0$. Therefore, 
the analytic inverse function theorem implies that there exists an
open ball $B$ centered at the origin, a neighborhood $V$ of $y$,
and an analytic diffeomorphism $\phi: B \to V$ so that 
$(\partial_j u) \circ \phi(z)=0$ if and only $z_j=0$. 
The open set $V$ contains some $x \in E$, and hence there
exists a neighborhood $W \subset V$  of $x$ such that $A \cap W$ 
is a hypersurface. It follows that $A \cap W=\phi({z:z_j=0}) \cap W$.  
In particular, for each $k$, the restriction of 
$(\partial_k u) \circ \phi(z)$ to $\{z: z_j =0\} \cap \phi^{-1}(W)$ 
vanishes identically. Since $(\partial_k u) \circ \phi(z)$ is real-analytic,
it vanishes on $B \cap \{z: z_j =0\}$, and hence 
$A \cap V= \{x: \partial_ju(x)=0\}$.  Therefore $y \in E$,
and $E$ is closed. 
\end{proof}

\begin{prop}
\label{prop:at-least-three-components}
Let $\Omega \subset \Rbb^d$ be a bounded open set
with Lipschitz boundary, and let $u: \Omega \to \Rbb$ 
be a nonconstant Neumann eigenfunction. 
Suppose that $M \subset \crit(u)$ is a hypersurface.
If $U$ is a connected component of $\Omega \setminus M$,
then $U \setminus u^{-1}(0)$ has at least two components.
\end{prop}

\begin{proof}
Since $\nabla u|_M=0$, the restriction $u|_U$ is a Neumann eigenfunction
on the bounded domain $U$. Integration by parts gives 
$ \lambda \int_U u = \int_U (\Delta u) \cdot 1 = 0$.
Since $u$ is nonconstant we have $\lambda \neq 0$, and 
hence $u$ takes both positive and negative values. 
\end{proof}

\begin{prop}
\label{prop:at-least-two-components}
Let $\Omega \subset \Rbb^d$ be simply connected, 
bounded, and open with Lipschitz boundary, and let $u: \Omega \to \Rbb$ 
be a nonconstant Neumann eigenfunction. If $\crit(u)$ contains 
an analytic hypersurface $M$, then $\Omega \setminus u^{-1}(0)$ has
at least three connected components.
\end{prop}

\begin{proof}
Proposition \ref{thm:crit-structure-higher} implies that 
$M$ is a smooth proper hypersurface, and so
Lemma \ref{lem:hypersurface-separate} below implies that 
$\Omega \setminus M$ has at least two connected components.
In fact, since $M$ is connected, $\Omega \setminus M$ has exactly 
two connected components $U_-$ and $U_+$  and $\oU_- \cap \oU_+ = M$.\footnote{
See e.g. \cite{Hirsch} Lemma 4.4.4.}
Proposition \ref{prop:at-least-three-components} implies that
$U_{\pm} \setminus u^{-1}(0)$ 
has (at least) two connected components.

By Corollary \ref{coro:no-intersect-nodal}, the 
hypersurface $M$ does not intersect $u^{-1}(0)$. 
Thus, since $M$ is connected, $M$ intersects at most 
one connected component $V$ of $\Omega \setminus u^{-1}(0)$. 
Thus, there exists a connected component $V_{\pm}$  
of $U_{\pm} \setminus u^{-1}(0)$ that is disjoint from $V$.

Let $\alpha:[-1,1] \to \Omega$ be a path that joins a point in 
$V_-$ to a point in $V_-$. Since $\alpha(\pm 1) \in U_{\pm}$ and
$U_-$ and $U_+$ are distinct connected components of $\Omega \setminus M$ with
$\oU_-\cap \oU_+ = M$, there exists $t$ so that $\alpha(t) \in M$.
In particular, since $M$ is a smooth manifold, there 
exists $s$ so that $\alpha(s) \in V$. Since $V$ is disjoint
from both $V_-$ and $V_+$, the components $V_+$ and $V_-$ are distinct.  
Hence $V_-$, $V_+$, and $V$ are three distinct components of 
$\Omega \setminus u^{-1}(0)$.
\end{proof}

The following is well-known. See \cite{Samelson} or Theorem 4.4.6 in \cite{Hirsch}.

\begin{lem}
\label{lem:hypersurface-separate}
Let $\Omega \subset \Rbb^d$ be a simply connected open set.
If $H \subset \Omega$ is a proper smooth hypersurface
then $\Omega \setminus H$ is not connected.
\end{lem}

\begin{proof}
Near $H$ we can find points $x_{\pm}$ and a smooth arc $\alpha:[-1,1] \to \Omega$ 
with $h(\pm 1)= x_{\pm}$ so that $\alpha^{-1}(H)=\{0\}$ and so that $\alpha'(0)$ 
does not lie in the tangent space to $H$. 

Suppose that $\Omega \setminus H$ were connected. 
Then there would exist an arc $\beta$ joining $x_+$ to $x_-$ so that $H \cap \beta = \emptyset$.
Since $\Omega$ is simply connected, there exists a homotopy 
$h:[-1,1]\times [0,1] \to \Omega$ with $h(t,0) =\alpha(t)$, $h(t,1)=\beta(t)$,
and $h(\pm 1,s) = x_{\pm}$ for each $s$.
By approximation, we may assume that $h$ is smooth and 
transverse to $H$ (see e.g. \cite{Hirsch} Theorem 3.2.1). 
It follows that $h^{-1}(H)$ is a 1-dimensional compact submanifold
of $S:=[-1,1] \times [0,1]$ whose boundary lies in $\partial S$. 
The point $(0,0)$ is the only point
that lies in both $h^{-1}(H)$ and $\partial S$. 
Hence the compact 1-manifold $h^{-1}(H)$
has exactly one boundary component, a contradiction. 
\end{proof}


\section{Invariant critical hypersurfaces}

In this section, we suppose that a hypersurface component of 
the critical set is invariant under local isometries, and 
show that this forces a solution to 
$\Delta u = \lambda u$ to depend only 
on the distance to the hypersurface.

Recall that a vector field is called a Killing field if and only 
if the local flow that it generates consists of local isometies. 
It follows that the Laplacian commutes with each Killing field.
Each Killing field on $\Rbb^d$ is either rotational or constant.  

Given an orientable hypersurface $M$, there exists a smooth function 
$r$ defined near $M$ so that $|r(x)|$ is the distance from $x$ to $M$. 
Let $\partial_r$ denote the gradient of $r$. Note that
$\partial_r$ is a unit vector field that is tangent to geodesics 
that meet $M$ orthogonally.  It is unique up to multiplication by $-1$.

\begin{lem}
\label{lem:commute}
If $X$ is a Killing field that is tangent to $M$,
then $[\partial_r, X] =0$.
\end{lem}

\begin{proof}
We claim that $\partial_r$ commutes with $X$.
Let $\phi_t$ denote the flow generated by $X$, and let  
$\psi_t$ denote the flow generated by $\partial_r$.  
Suppose $x \in M$. Then for each sufficiently small
$t$ we have $\dist(x,\psi_{t}(x))= t$ and since $\phi_s$ 
is as isometry we have $\dist( \phi_s(x), \phi_s(\psi_t(x)))=t$
for small $s$. 
The unique geodesic that realizes the latter distance
is a flow line of $\partial_r$ with length $t$ and hence
$\psi_t(\phi_s(x))= \phi_s(\psi_t(x))$ for all small $s$, $t$.
The claim follows.
\end{proof}

\begin{prop}
\label{prop:Killing-vanish}
Let $\Omega \subset \Rbb^d$ be open and connected, 
and let $M \subset \Omega$ be an oriented hypersurface.
Suppose that $X$ is a Killing field on $\Omega$ that is tangent to $M$. 
If $u: \Omega \to \Rbb$ satisfies $\Delta u = \lambda \cdot u$  and $\nabla u$ 
vanishes on $M$, then $X u \equiv 0$ on $\Omega$.
\end{prop}

\begin{proof}
Because $X$ is a Killing field, the function $Xu$ is a Laplace eigenfunction.  
Since $\nabla u=0$, we have $Xu(x) =0$ and $\partial_r u(x)=0$ for each 
$x \in M$.  Therefore, Lemma \ref{lem:commute} 
implies $\partial_r(Xu)(x) = X(\partial_r u)(x)=0$
for each $x \in M$.
Thus, since $Xu$ vanishes on the hypersurface $M$ we find that $\nabla (Xu)=0$.
Therefore, the claim follows from Lemma \ref{lem:Cauchy-vanish}.
\end{proof}

\begin{lem}
\label{lem:Cauchy-vanish}
Let $\Omega \subset \Rbb^d$ be an open set and let
$u : \Omega \to \Rbb$ satisfy $\Delta u = \lambda \cdot u$.
If $\sing(u)$ contains a hypersurface, then $u \equiv 0$.
\end{lem}

\begin{proof}
Both $u$ and its normal derivative vanish along 
the hypersurface. The zero function lies in the 
kernel of $\Delta-\lambda \cdot I$.
Therefore the uniqueness theorem of 
Holmgren\footnote{See, for example, Theorem 6.4.3 in \cite{Taylor}.
One can also use the unique continuation theorem of H\"ormander
\cite{Hormander}  \cite{Tataru}.} 
implies that $u \equiv 0$.
\end{proof}

Let $G$ be a group of isometries. 
We will say that the action of $G$ is {\em codimension 1}
if the typical orbit is a hypersurface. 
For example, the standard action of $SO(n-k) \times \Rbb^{k-1}$ on 
$\Rbb^{n-k} \times \Rbb^{k-1}$ is codimension 1.

\begin{thm}
\label{thm:invariant}
Let $\Omega \subset \Rbb^d$ be open and connected, 
and let $M \subset \Omega$ be a hypersurface that is contained
in the typical orbit of a codimension one isometric group action.
If $M \subset \crit(u)$ and $\Delta u = \lambda u$, then 
$u$ is invariant under the action.
\end{thm}

\begin{proof}
Let $x \in M$. Since $G \cdot x$ contains $M$ there are Killing fields 
$X_1, \ldots, X_{n-1},$ such that the tangent space at $x$ is spanned by 
$\{X_1(x), \ldots, X_{n-1}(x)\}$. Proposition  \ref{prop:Killing-vanish}
implies that $X_j u \equiv 0$ for each $j$. Since the $X_j$ span
the tangent space of an orbit, the function $u$ is constant 
on the orbit.
\end{proof}

\begin{eg} 
\label{eg:hyperplane}
Suppose that the hypersurface $M$ belongs to a hyperplane in $\Rbb^d$.
The hyperplane is the orbit of an isometric action of $\Rbb^{n-1}$
via translations parallel to the hyperplane.
Theorem \ref{thm:invariant} implies that $u = v \circ r$
for some $v: \Rbb \to \Rbb$.
Since $\Delta u = \lambda u$ and $\nabla u=0$ along $M$, 
we find that $v$ is a multiple of $ t \mapsto \cos(\sqrt{\lambda} \cdot t )$.
\end{eg}

\begin{eg}
Suppose that the hypersurface $M$ belongs to an $n-1$ dimensional 
sphere in $\Rbb^d$. The sphere is the orbit of an isometric action of 
the orthogonal group $O(n)$. Theorem \ref{thm:invariant} implies 
that $u = v \circ r$ for some $v: \Rbb \to \Rbb$
that satisfies an explicit second order ordinary differential
equation. In particular, $v$ is a Bessel function.
\end{eg}


\section{Critical sets of second Neumann eigenfunctions}
\label{sec:2nd-Neumann}

In this section, we suppose that $u$ is a second Neumann 
eigenfunction on a simply connected domain, and derive
consequences.  

\begin{prop}
\label{coro:lem-2nd-no-hypersurface}
Let $\Omega \subset \Rbb^d$ be a simply connected, 
bounded, open set with Lipschitz boundary, 
and let $u: \Omega \to \Rbb$ be a nonconstant second Neumann eigenfunction.
Then $\crit(u)$ does not contain a hypersurface.
\end{prop}

\begin{proof}
Courant's nodal domain theorem implies that the set 
$\Omega \setminus u^{-1}(0)$ has exactly two components.
If $\crit(u)$ were to contain an hypersurface,
then Proposition \ref{prop:at-least-two-components}
would imply that $\Omega \setminus u^{-1}(0)$ has at least three components.
\end{proof}

If a portion of the boundary of $\Omega$ is sufficiently regular, then 
each Neumann eigenfunction extends in a $C^1$ fashion to a portion 
of the boundary,
and hence one can extend the notion of critical point 
of $u$ on the boundary.  We will let $\ocrit(u)$ denote 
the set of critical points of such an extension if it exists.

\begin{prop}
\label{prop:face-crit-rect}
Let $\Omega \subset \Rbb^d$ be a bounded, simply connected domain
with piecewise analytic boundary,
and let $u: \Omega \to \Rbb$ be a second Neumann eigenfunction. 
If $\ocrit(u)$ contains a hypersurface $M$ that lies in a hyperplane $H$, 
then there exists a domain $D \subset H$ such that 
$\Omega$ is the intersection of the cylinder over $D$ and the
convex hull of $H \cup H'$ where $H'$ is a hyperplane parallel 
to $H$ such that $\dist(H,H')= \pi/ \sqrt{\lambda}$.
\end{prop}

\begin{proof}
Proposition \ref{coro:lem-2nd-no-hypersurface} implies that
the hypersurface $M$ cannot lie in the interior of $\Omega$.
Hence $M$ lies $\partial \Omega$. Because $u$
satisfies Neumann conditions, the eigenfunction $u$ extends
via reflection to a neighborhood $U$ of a point of $H$, and 
hence we are in the situation of Example \ref{eg:hyperplane}. In particular,
$u$ is a multiple of $x \mapsto \cos(\sqrt{\lambda} \cdot \dist(x,U))$ 
Therefore, since $u$ satisfies Neumann conditions, if the vector $v$
with footpoint $x$ is tangent to $\partial \Omega$, then either 
$v$ is orthogonal to $H$ or $\dist(x, H)= k\pi/\sqrt{\lambda}$ for some $k \in \Zbb$.
Proposition \ref{coro:lem-2nd-no-hypersurface} implies that at most two 
distinct integers $k$ can appear and that they are successive integers.
The claim follows.
\end{proof}

A similar argument gives the following. 

\begin{prop}
\label{prop:face-crit-spherical}
Let $\Omega \subset \Rbb^d$ be a bounded, simply connected domain
with piecewise analytic boundary,
and let $u: \Omega \to \Rbb$ be a second Neumann eigenfunction. 
If $\ocrit(u)$ contains a hypersurface $M$ that lies in a hypersphere $S$, 
then there exists a domain $D \subset S$ such that $\Omega$ 
is the intersection of the cone on $D$ and a spherical shell 
whose boundary contains $S$. 
\end{prop}

The following finishes the proof of Theorem \ref{thm:rectangle}
in the introduction. 

\begin{thm}
\label{thm:rectangle-interior}
Let $P \subset \Rbb^2$ be a bounded polygonal domain, 
and let $u: P \to \Rbb$
be a second Neumann eigenfunction. If $\ocrit(u)$ is not finite,
then $P$ is a rectangle. 
\end{thm}

\begin{proof}
Via reflection across its sides, one may extend $u$ to an eigenfunction
$\tu$ defined on an open set $\Omega$ that contains $\oP \setminus V$.
Suppose that $\crit(u)$ is not finite, and let 
$p$ be an accumulation point.  
Lemma \ref{lem:no-vertex-limit} below implies that $p$ is not a vertex of $P$. 
Thus $p$ lies in the interior of $\Omega$ and hence $p$ 
is a critical point of $\tu$. Because $|\nabla \tu|^2$ 
is real-analytic, the critical set is a locally finite graph.\footnote{See,
for example, the proof of Proposition 5 in \cite{Otal-Rosas}.}
Thus, since $p$ is not an isolated critical point, there exists 
an arc in $\crit(\tu)$. Proposition \ref{coro:lem-2nd-no-hypersurface}
implies that this arc cannot intersect the interior of $P$. 
Because $\tu$ is defined on the exterior of $P$ via by reflection
in the sides, the arc also cannot intersect the exterior of $\oP$.
Hence the arc lies in a side of $P$, and Proposition \ref{prop:face-crit-rect}
implies that $P$ is a rectangle.
\end{proof}

The following Lemma generalizes Proposition 5.6 in \cite{Annals},
a result that assumed the convexity of the polygonal domain.

\begin{lem}
\label{lem:no-vertex-limit}
Let $P$ be a bounded, simply connected, polygonal
domain\footnote{In particular, $P$ is an open set.} 
and let $u: P \to \Rbb$ be a second Neumann eigenfunction.
No vertex $v$ of $P$ is an accumulation point of $\crit(u)$. 
\end{lem}

\begin{proof}
Suppose to the contrary that a sequence of critical points converges to a vertex.  
By applying a planar isometry to the polygon, we may assume that the vertex is 
located at the origin and that the intersection of $P$ with a disc neighborhood
of the origin has the form 
$\{r e^{i \theta} \in \Cbb\, :\, 0 < \theta < \beta,\, 0 < r <  r_0  \}$
where $\beta \neq \pi$ is the angle of the vertex.

The Neumann eigenfunction $u$ has the Fourier-Bessel expansion 
\begin{equation}  
\label{eqn:Bessel-expansion-II}
u \left(r e^{i\theta} \right)~ 
=~ 
\sum_{k=0}^\infty\, c_n \cdot r^{k \cdot \nu} \cdot g_{k \cdot\nu}\left(r^2\right) \cdot
 \cos \left(k \cdot \nu \cdot \theta \right).
\end{equation}
where $\nu:= \pi/\beta$, $g_{\nu}(r^2):= r^{-\nu} \cdot J_{\nu}(\sqrt{\lambda} \cdot r)$,
and $J_{\nu}$ is the standard Bessel function (see e.g. \S 4 \cite{Annals}). From the series expansion for $J_{\nu}$ 
\cite{Lebedev}, one finds that the function $g_{\nu}$
is entire and $\partial_r^k g_{\nu}(0) \neq 0$ for $k \in \Nbb$.

Because $P$ is simply connected, Corollary 5.3 in \cite{Annals} 
implies that either $c_0$ or $c_1$ is nonzero. Hence we will derive 
a contradiction by showing that both $c_0$ and $c_1$ equal zero.

Differentiation of (\ref{eqn:Bessel-expansion-II}) 
with respect to $\theta$ gives
\begin{equation}
\label{eqn:diff-theta}
\partial_{\theta} u \left(r \cdot e^{i\theta}\right)~
=~
-\nu \cdot r^{\nu} \cdot \sin(\nu \cdot \theta) \cdot 
\left( c_1 \cdot g_{\nu}(r^2) ~
+~
O(r^{\nu}) \right). 
\end{equation}
Let $r_n e^{i \theta_n}$ denote the sequence of critical points 
that converges to the origin. 
Then from (\ref{eqn:diff-theta}) we have $c_1 \cdot g_{\nu}(r_n^2)= O(r_n^{\nu})$.
Thus, since $\nu>0$ and $g_{\nu}(0)\neq 0$, it follows that $c_1=0$.

Let $k$ be the smallest integer greater then 1 
such that $c_k \neq 0$. Since $c_1=0$, differentiation  of (\ref{eqn:Bessel-expansion-II}) 
with respect to $r$ gives 
\begin{equation}
\label{eqn:r-derivative}
\partial_r u\left( r \cdot e^{i \theta}\right)~
=~
c_0 \cdot 2r \cdot g_0'(r^2) ~
+~
c_k \cdot k\nu \cdot r^{k\nu-1} 
\cdot g_{k\nu}(r^2) \cdot\cos(k\nu \theta)~
+~
O \left(r^{k\nu+1} \right)~
+~
O \left(r^{(k+1)\nu-1} \right).
\end{equation}
If $re^{i\theta}$ is a critical point, then the 
left hand side vanishes. As $r$ tends to zero, 
one of the first two terms on the right dominates.
If $k \nu -1 > 1$, then 
since $r_n e^{i \theta_n}$ is a critical point 
and $g_0'(0) \neq 0$ we find
that $c_0 \cdot 2r_n \cdot g_{0}'(r_n^2) = o(r_n)$ and 
hence $c_0 =0$. If $k \nu -1 < 1$, then 
we find that $c_k \cdot \cos(k \nu \theta_n)= o(1)$.
Since $c_k \neq 0$, we have $\cos(k \nu \theta_n) \to 0$,
and hence $|\sin(k \nu \theta_n)| \to 1$. 
On the other hand, since
\begin{equation}
\label{eqn:diff-theta2}
\partial_{\theta} u \left(r \cdot e^{i\theta}\right)~
=~
- c_k \cdot \nu \cdot r^{k\nu} \cdot g_{k\nu}(r^2)  \cdot \sin(\nu k \theta) 
+~
O \left(r^{(k+1)\nu} \right). 
\end{equation}
we find that either $c_k =0$ or $\sin(k \nu\theta_n) \to 0$, a contradiction in 
either case. 

If $k \nu -1 =1$, then since $k \geq 2 $ and $0 <\beta < 2 \pi$ with 
$\beta \neq \pi$, we find that $k=3$ and  $\nu = 2/3$ ($\beta = 3 \pi/2$).
If we define
\begin{equation}  
\label{eqn:Bessel-expansion-II-third}
f(s, t)~ 
=~ 
\sum_{k=0}^\infty\, c_k \cdot s^{2k} 
\cdot g_{ \frac{2k}{3}}\left(s^6\right) \cdot
 \cos \left( k \cdot t \right)
\end{equation}
then from (\ref{eqn:Bessel-expansion-II}) we find that
$u(r \cdot e^{i\theta})= f\left(r^{\frac{1}{3}}, \frac{2}{3} \theta\right)$. 
It is known that for $\mu>0$, the Bessel function $J_{\mu}$ 
is nonnegative and increasing on the interval 
$[0, \mu]$.\footnote{See, for example, \S 15.3 \cite{Watson}.} 
It follows that for each $k >0$, the function 
$s \mapsto s^{2k} \cdot g_{\frac{2k}{3}}(s^6)$ 
is nonnegative and increasing on the interval 
$[0, (2/3 \sqrt{\lambda})^{\frac{1}{3}}]$.
Since $u$ is real-analytic on $\Omega$, there
exists $s_0 < (2/3 \sqrt{\lambda})^{\frac{1}{3}}$, 
such that $t \mapsto f(s_0,t)$ is 
continuously differentiable on $[0, \pi]$. In particular, 
$\partial_t f(s_0,t)$ belongs to $L^2([0, \pi))$.
Hence by Bessel's inequalilty, the sum
\[
\sum_{k=0}^\infty~ k^2 \cdot |c_k|^2 \cdot s_0^{4k} 
\cdot \left|g_{\frac{2k}{3}}\left(s_0^6\right) \right|^2 
\]
is finite. Thus, since $\sum_k k^{-2}$ is finite, it follows 
from the Cauchy-Schwarz inequality that
\[
\sum_{k=0}^\infty~ |c_k| \cdot s_0^{2k} 
\cdot \left| g_{\frac{2k}{3}}\left(s_0^6\right) \right| 
\]
is finite. Since for $k>0$ the function
$s \mapsto s^{2k} \cdot g_{\frac{2k}{3}}(s^6)$ 
is nonnegative and increasing on the interval 
$[0, s_0]$, we find that the series in (\ref{eqn:Bessel-expansion-II-third})
is uniformly convergent on $[0,s_0]$, and thus by symmetry it is uniformly
convergent on $[-s_0,s_0]$.
It follows that $f$ is analytic on $[-s_0, s_0] \times [0,\pi]$.
In particular, $f$ has at most finitely many critical points 
on $[0, s_0] \times [0,\pi]$.  Therefore $u$ does not have 
a sequence of critical points converging to the vertex. 
\end{proof}


{\it Addresses:}\\
Department of Mathematics, Indiana University, Bloomington, IN, USA\\
School of Mathematics, Tata Institute of Fundamental Research, Mumbai, India.

{\it Emails:}\\
cjudge2@gmail.com\\
sugatam@math.tifr.res.in

\end{document}